\documentclass[12pt]{amsart}
\usepackage{amssymb}
\usepackage{verbatim}
\usepackage[usenames]{color}
\usepackage{hyperref}
\usepackage{url}
\usepackage[all]{xy}
\usepackage{hyperref}

\newtheorem{theorem}{Theorem}[section]
\newtheorem*{theorem*}{Theorem}
\newtheorem{proposition}[theorem]{Proposition} 
\newtheorem{observation}[theorem]{Observation} 
\newtheorem{lemma}[theorem]{Lemma}
\newtheorem{cor}[theorem]{Corollary}
\newtheorem{fact}[theorem]{Fact}

\theoremstyle{definition}
\newtheorem{definition}[theorem]{Definition}

\newtheorem{question}[theorem]{Question}

\theoremstyle{remark}
 
\newtheorem{remark}[theorem]{Remark}

\newcommand{\defemp}{\textit}

\newcommand{\forces}{\Vdash}
\newcommand{\zf}{\textrm{ZF}}
\newcommand{\ad}{\textrm{AD}}

\newcommand{\zfc}{\textrm{ZFC}}
\newcommand{\hod}{\textrm{HOD}}

\newcommand{\ch}{\textnormal{CH}}

\newcommand{\dom}{\textnormal{Dom}}

\newcommand{\mc}{\mathcal}
\newcommand{\mbb}{\mathbb}

\newcommand{\p}{\mathcal{P}}

\newcommand{\baire}{{^\omega \omega}}
\newcommand{\bairenodes}{{^{<\omega}\omega}}

\title[Coding with Help and Amalgamation Failure]
 {Generic Coding with Help \\ and Amalgamation Failure }
\author{Sy-David Friedman}
\author{Dan Hathaway}

\address{
Sy-David Friedman \\
Kurt G\"odel Research Center for Mathematical Logic \\
University of Vienna \\
Vienna, Austria}
\email{sdf@logic.univie.ac.at}

\address{
Dan Hathaway \\
Mathematics Department \\
University of Vermont\\
Burlington, VT 05401, U.S.A.}
\email{Daniel.Hathaway@uvm.edu}

\thanks{
The first author was partially supported
  by the FWF (Austrian Science Fund) grant \#28157.
\\
The second author proved some of these results
 during the September 2012 Fields Institute Workshop
 on Forcing while being supported by the Fields Institute.
Work was also done by the second author while under NSF grant DMS-0943832.
}

\begin{document}

\begin{abstract}
We show that
 if $M$ is a countable transitive model
 of $\zf$ and
 if $a,b$ are reals not in $M$,
 then there is a $G$ generic over $M$
 such that $b \in L[a,G]$.
We then present several applications
 such as the following: if $J$ is any
 countable transitive model of $\zfc$
 and $M \not\subseteq J$
 is another countable transitive model of $\zfc$
 of the same ordinal height $\alpha$, then there is
 a forcing extension $N$ of $J$ such that
 $M \cup N$ is not included in any
 transitive model of $\zfc$ of height $\alpha$.
Also, assuming $0^\#$ exists,
 letting $S$ be the set of reals generic over $L$,
 although $S$ is disjoint from the Turing cone above $0^\#$,
 we have that for any non-constructible real $a$,
 $\{ a \oplus s : s \in S \}$ is cofinal in the Turing degrees.
\end{abstract}

\maketitle

\section{Introduction}

If $0^\#$ exists, it is not in any (set) forcing extension of $L$.
On the other hand,
 Mostowski showed that for any real $x$,
 there are reals $g_1, g_2$ both Cohen generic over $L$
 such that $x$ is computable from the Turing join of $g_1$ and $g_2$,
 written $x \le_T g_1 \oplus g_2$
 (see \cite{blockchains} for a proof).

In this paper we investigate the following question
 (assuming $0^\#$ exists):
 given an arbitrary $g_1 \in \baire$ not in $L$,
 is there a real $g_2$ generic over $L$ such that $0^\# \le_T g_1 \oplus g_2$?
We will see that the answer is yes.
Although there is a limit
 to what reals are generic over $L$,
 there is \textit{no limit} to what reals are constructible from
 a fixed non-constructible real and
 a real that is generic over $L$.
Here is the general formulation:
\begin{definition}
Let $M$ be a countable transitive model of $\zfc$.
Let $\mathbb{P} \in M$ be a poset.
A real $\bar{a} \in \baire$ is $(\mathbb{P},M)$-\defemp{helpful}
 iff for any $x \in \baire$,
 there is a $G$ that is $\mathbb{P}$-generic over $M$ such that
 $x \in L(\bar{a},G)$.
\end{definition}

Now fix a countable transitive model $M$ of $\zfc$.
Let $\mathbb{C}$ be Cohen forcing.
\begin{itemize}
\item[1)]
 Fix $\mathbb{P} \in M$.
 No real $\bar{a} \in M$ is $(\mathbb{P},M)$-helpful:
 if $x \in \baire$ codes the ordinal $\mbox{Ord} \cap M$,
 then $x \not\in L(\bar{a},G)$ for any
 $G$ that is $\mathbb{P}$-generic over $M$.
\item[2)]
 Every real Cohen generic over $M$
 is $(\mathbb{C},M)$-helpful
 (see Corollary 5.5 of \cite{blockchains}).
\item[3)] Miha Habi\u{c} (unpublished) and the first author
 (see \cite{hyperuniverse} just after ``nodes of compatibility'')
 have independently shown that
 every real unbounded over $M$
 is $(\mathbb{C},M)$-helpful.
\item[4)] The first author has shown that every real Sacks generic over $M$
 is $(\mathbb{C},M)$-helpful
 (unpublished).
\item[5)] The central result of this paper
 (Theorem~\ref{main_thm})
 is that \textit{every} real not in $M$
 is $(\mathbb{H},M)$-helpful,
 where $\mathbb{H}$ is ``Tree-Hechler'' forcing.
\item[6)] The question of whether every real
 not in $M$ is
 $(\mathbb{C},M)$-helpful remains open.
\end{itemize}

\begin{definition}
The forcing $\mbb{H}$, called Tree-Hechler forcing, consists of all
 trees $T \subseteq {^{<\omega}\omega}$ such that
 for all $t \sqsupseteq \mbox{Stem}(T)$ in $T$,
 $$\{ z \in \omega :
 t ^\frown z \not\in T \} \mbox{ is finite. }$$
The ordering is by inclusion.
\end{definition}

That is, a condition in Tree-Hechler forcing
 has cofinite splitting beyond its stem.

Consider a tree $T \subseteq \bairenodes$
 and a node $t \in T$.
By a \textit{successor} of $t$ we always mean
 some $t ^\frown z \in T$ for $z \in \omega$.
By $T \restriction t$ we mean the set of all
 $s \in T$ that are comparable to $t$.
$\mbox{Stem}(T)$ is the longest element of $T$
 that is comparable with all other elements of $T$.

Let $M$ be a transitive model of $\zf$ and
 suppose $G$ is $\mbb{H}^M$-generic over $M$.
Let $g = \bigcup \bigcap G$.
That is, $g : \omega \to \omega$
 is the union of all the stems
 of the trees $T \in G$.
The set $G$ can be recovered from $g$ (and $M$).
We will treat $g : \omega \to \omega$
 as the object which is encoding information.

The poset $\mbb{H}$ is $\sigma$-centered,
 because any two conditions with the same stem
 are comparable.
Thus, $\mbb{H}$ is c.c.c.
Combining this with the fact that
 $|\mbb{H}| = 2^\omega$,
 we have the following:
 there are only $2^\omega$ maximal antichains
 in $\mbb{H}$.
So, if $M$ is a transitive model of $\zfc$
 and $(2^\omega)^M$ is countable,
 then there is an $\mbb{H}^M$-generic
 over $M$.


The forcing $\mbb{H}$ is discussed in
 \cite{Palumbo}, along with other versions of Hechler forcing,
 where it is called $\mbb{D}_{tree}$.
A key ingredient for us is that
 $\mathbb{H}$ admits a ``rank analysis'' of its dense sets
 (see Definition~\ref{reachability} and
 Lemma~\ref{dense_implies_reachable}).
In \cite{even}, J\"org Brendle and Benedikt L\"owe
 carry out a rank analysis of $\mbb{H}$.
The original rank analysis of a Hechler-like forcing
 was done by James Baumgartner and Peter Dordal in
 \cite{baum} for the non-decreasing function version of Hechler forcing
 (although we discovered ``reachability'' independently of these works).

Here is our main result:
\begin{theorem}[Generic Coding with Help]
\label{main_thm}
Let $M$ be a transitive model of $\zf$ such that
 $\p^M(\mbb{H}^M)$ is countable.
Then given any $\bar{a}, x \in \baire$
 such that $\bar{a} \not\in M$,
 there is a $G$ that is $\mbb{H}^M$-generic over $M$
 such that
 $x \le_T \bar{a} \oplus (\bigcup \bigcap G)$.
\end{theorem}
Here, $\bar{a}$ is the ``help'' which is being used
 to code $x$.
Theorem~\ref{main_thm}
 has several interesting applications,
 which we will present first.
Then for completeness we will include
 a proof of Theorem~\ref{main_thm}.

One striking application is
 Theorem~\ref{amalg_thm},
 which shows that given two distinct
 countable transitive models $M,J$ of $\zfc$
 of the same height
 (meaning $\mbox{Ord} \cap M = \mbox{Ord} \cap J$),
 there is a forcing extension of one
 which does not amalgamate with the other
 (where two models of the same height $\alpha$
 are said to amalgamate iff they are both included
 in a countable transitive model $W$ of $\zfc$
 of height $\alpha$).
This answers a question posed by the first author \cite{hyperuniverse}
 concerning the Hyperuniverse Program:
 the model $L_\alpha$ is the only node of compatibility of height $\alpha$
 (see our discussion after Corollary~\ref{no_notrivial_nodes_of_combat}).

Theorem~\ref{main_thm} is a consequence of
 Lemma~\ref{main_lemma}, the ``Main Lemma''.
However, this was not the Main Lemma's original purpose
 in the literature.
This lemma originated from
 the second author's thesis \cite{Hathaway3}
 where it appeared in a game theoretic form
 that does not explicitly refer to forcing.
In that version,
 Players I and II play to build a descending
 sequence through $\mathbb{H}$,
 where Player I makes $\le$-extensions but
 Player II makes $\le_A$-extensions (to be defined later).
The goal of this game was to prove
 results like Proposition~\ref{original_application}.
In \cite{Hathaway1} and \cite{Hathaway2}
 such results are proved,
 and the current version of the Main Lemma
 appears in \cite{Hathaway1}.
We want to emphasize that the Main Lemma may have
 applications other than
 Theorem~\ref{main_thm} and
 Proposition~\ref{original_application}.

After this paper was published, we
 observed that our theorem has an application
 to coding the universe.
It was known that the universe can be coded
 by a real that is class generic over $V$
 (see \cite{coding_the_universe}).
By combining both our theorem and that
 coding the universe result,
 it follows that for any inner model $M$
 and for any set of ordinals $\bar{a}$ not in $M$,
 we have $V \subseteq L[G][\bar{a}]$ for some $G$
 (in a class forcing extension of $V$)
 that is \textit{set} generic over $M$.
We add this result in an addendum at
 the end of this paper.


\section{Amalgamation failure for C.T.M.'s}

The Generic Coding with Help Theorem
 implies in a strong way that c.t.m.'s
 (countable transitive models)
 of $\zfc$
 of the same ordinal height
 cannot be amalgamated:

\begin{theorem}
\label{amalg_thm}
Let $J$ be a c.t.m.\ of $\zfc$
 of ordinal height $\alpha < \omega_1$.
Let $M \not\subseteq J$ be another c.t.m.\
 of $\zfc$ of height $\alpha$.
Then there is a forcing extension $N$ of $J$
 such that $M \cup N$ is not included in
 any c.t.m.\ of $\zfc$ of height $\alpha$.
\end{theorem}
\begin{proof}
Fix $\lambda < \alpha$
 and $x \subseteq \lambda$ such that $x \in M - J$.
This is possible because
 $J$ and $M$ are models of $\zfc$
 and $M \not\subseteq J$.
That is, following the proof of
 Theorem 13.28 in \cite{Jech},
 first fix $X \in M - J$.
Now let $x \in M$ be a bounded subset of
 $\mbox{Ord} \cap M = \alpha$
 such that $X$ is in any transitive model of $\zfc$
 which contains $x$:
 such an $x$ can be formed by first bijecting
 the transitive closure $tc(\{X\})$
 of $\{X\}$ with an ordinal $\lambda' < \alpha$,
 and then encoding the binary relation $\in \restriction tc(\{X\})$
 as a subset of $\lambda' \times \lambda'$,
 and then encoding that binary relation by a single set
 $x \subseteq \lambda$ for some $\lambda < \alpha$.
Such an $x$ cannot be in $J$.

Let $g_0'$ and $g_0''$ be mutually
 $\mbox{Col}(\omega,\lambda)$-generic over $J$.
Since they are mutually generic,
 $J[g_0'] - J$ and $J[g_0''] - J$ are disjoint.
Let $g_0$ be one of $g_0'$ or $g_0''$ such that
 $x \not\in J[g_0]$.

Now $g_0$ codes a surjection from $\omega$ to $\lambda$.
Let $\tilde{x} \subseteq \omega$ be induced from
 this surjection and $x$.
By this we mean if $W$ is any transitive model
 of $\zfc$ with
 contains $g_0$, then $x \in W$ iff $\tilde{x} \in W$.
Now $\tilde{x} \not\in J[g_0]$.

Let $y \in \baire$ be a real that codes a well-ordering
 of $\omega$ of order type $\alpha$
 (so $y$ cannot be in any c.t.m.\
 of $\zfc$
 of height $\alpha$).
By Theorem~\ref{main_thm},
 let $g_1$ be $\mbb{H}^{J[g_0]}$-generic over $J[g_0]$
 such that
 $$y \le_T \tilde{x} \oplus (\bigcup \bigcap g_1).$$

Let $N = J[g_0][g_1]$.
Now suppose, towards a contradiction,
 that there is some transitive model
 $W \supseteq M \cup N$ of $\zfc$ of
 ordinal height $\alpha$.
Because $x \in M \subseteq W$
 and $g_0 \in N \subseteq W$,
 we have $\tilde{x} \in W$.
But also $g_1 \in N \subseteq W$,
 so $y \in W$, which is impossible.
\end{proof}

We say two c.t.m.'s $N,M$ of $\zfc$
 of height $\alpha$
 are \textit{compatible}
 iff there is a c.t.m.\ $W$ of $\zfc$
 of height $\alpha$ such that
 $N \cup M \subseteq W$.

\begin{cor}
\label{no_notrivial_nodes_of_combat}
Given any
 two distinct c.t.m.'s of $\zfc$ of the same height,
 there is a forcing extension of one that is not compatible with the other.
\end{cor}

The first author asked (see \cite{hyperuniverse})
 if for a given $\alpha < \omega_1$,
 whether $L_\alpha$ was the only c.t.m.\ of $\zfc$ of height $\alpha$
 that was compatible with every c.t.m.\ of $\zfc$ of height $\alpha$
 (that is, whether $L_\alpha$ was the only
 \textit{node of compatibility}
 of height $\alpha$
 in the Hyperuniverse).
Now we see the answer is yes:
 If $M \not= L_\alpha$ is a c.t.m.\ of $\zfc$ of height $\alpha$,
 then $M$ is not compatible with a certain forcing extension of $L_\alpha$.


\begin{remark}
\label{folklore}
Mostowski's result in the introduction
 was used by him for a result about amalgamation
 (see \cite{blockchains}):
Let $J$ be a c.t.m.\ of $\zf$
 of ordinal height $\alpha < \omega_1$.
Let $x$ be a real which codes $\alpha$.
Let $c_1, c_2$ be two reals Cohen generic over $J$
 such that $x \le_T c_1 \oplus c_2$.
Then $J[c_1]$ and $J[c_2]$
 are not compatible.
\end{remark}

\section{A complex set disjoint from a Turing cone}

As mentioned before,
 if $0^\#$ exists
 (or even just $\omega_1$ is inaccessible in $L$),
 then given any real $x$,
 there are two Cohen generics $s_1, s_2$ over $L$
 such that $x \le_T s_1 \oplus s_2$.
So, let $S$ be the complement of the Turing cone above
 $0^\#$
 (the Turing cone above $a \in \baire$
 is the set of all $b \in \baire$ such that
 $b \ge_T a$).
Every real generic over $L$ is in $S$.
Now $S$ is small in one sense, because it is
 disjoint from a Turing cone.
But it is large in another sense, because
 $\{ s_1 \oplus s_2 : s_1, s_2 \in S\}$
 is cofinal in the Turing degrees.
We get a variation of this phenomenon 
 using
 the Generic Coding with Help Theorem (\ref{main_thm}).
Let $[x]$ denote the Turing degree of $x \in \baire$.

\begin{proposition}
Assume $0^\#$ exists.
Let $S \subseteq \baire$ be the set of all reals
 of the form $s = \bigcup \bigcap G$ for some
 $G$ that is $\mbb{H}^L$-generic over $L$.
The set $S$ is disjoint from the Turing cone above $0^\#$.
On the other hand
 for any real $\bar{a}$ not in $L$,
 the set $S^* := \{ [\bar{a} \oplus s] : s \in S \}$
 is cofinal in the Turing degrees.

Also, if $x$ is any real such that $x \ge_T \bar{a}$
 and $x$ computes a length $\omega$ enumeration
 of $\mathbb{R} \cap L$,
 then $[x] \in S^*$
 (so $S^*$ contains a Turing cone).
\end{proposition}
\begin{proof}
It is well known that no generic extension of $L$
 contains $0^\#$.
Hence, $0^\#$ is not Turing reducible to any
 $s = \bigcup \bigcap G$ for a $G$ that is
 $\mbb{H}^L$-generic over $L$.
That is, $S$ is disjoint from the Turing cone
 above $0^\#$.

Now fix a real $\bar{a}$ not in $L$.
Pick any $x \in \baire$.
By Theorem~\ref{main_thm}
 there is some $G$ that is $\mbb{H}^L$-generic over $L$
 such that letting $s = \bigcup \bigcap G$,
 we have $x \le_T \bar{a} \oplus s$.
Hence, $S^*$ is cofinal in the Turing degrees.

For the last part,
 again fix a real $\bar{a} \not\in L$
 and let $x \ge_T \bar{a}$ be a real which
 computes a length $\omega$ enumeration of $\mathbb{R} \cap L$.
There is some $G$ that is $\mathbb{H}^L$-generic over $L$
 such that letting $s = \bigcup \bigcap G$,
 we have $x \le_T \bar{a} \oplus s$.
However,
 by the proof of Theorem~\ref{main_thm},
 fix an $s$ like this that
 can be built using $\bar{a}$, $x$,
 and a length $\omega$ enumeration of $\mathbb{R} \cap L$
 (using that the dense subsets of $\mathbb{H}^L$ in $L$
 are coded by reals in $L$).
So we can have $s \le_T \bar{a} \oplus x$.
We now have
 $$x \le_T
 \bar{a} \oplus s \le_T
 \bar{a} \oplus (\bar{a} \oplus x) \le_T
 \bar{a} \oplus x =
 x,$$
 so $x =_T \bar{a} \oplus s$,
 and so $[x] \in S^*$.
\end{proof}

\section{Larger sets are generically generic}

The Generic Coding with Help Theorem shows that
 reals not in $M$ are ``helpful''.
The following theorem shows that \textit{any set of ordinals
 not in $M$
 is helpful}, provided $M$ contains the supremum of the
 set of ordinals and that we pass to an outer model of $V$
 in which a large enough cardinal has become countable.

\begin{theorem}
\label{generically_generic}
Let $M$ be a transitive model of $\zf$.
Let $\lambda$ be a cardinal such that $\lambda \in M$.
Let $\mbb{P} = (\mbox{Col}(\omega,\lambda) * \mbb{H})^M$.
Let $\tilde{V}$ be an outer model of $V$ in which
 $\p^M(\mbb{P})$ is countable.
Let $X \in \p^{\tilde{V}}(\lambda)$.
Let $A \in \p^{\tilde{V}}(\lambda) - M$.
Then there is a $G$ in $\tilde{V}$ such that
\begin{itemize}
\item[1)] $G$ is $\mbb{P}$-generic over $M$,
\item[2)] $X \in L(A,G)$.
\end{itemize}
\end{theorem}
\begin{proof}
Using the same mutual generic technique
 as in the second paragraph of
 the proof of Theorem~\ref{amalg_thm},
 let $g_0 \in \tilde{V}$
 be $\mbox{Col}(\omega,\lambda)$-generic over $M$
 so that $A \not\in M[g_0]$.
Let $\tilde{a} \in \baire$ be such that
 for every transitive model $N$ of $\zf$ such that
 $g_0 \in N$,
 we have
 $A \in N$ iff
 $\tilde{a} \in N$.
Now $\tilde{a} \not\in M[g_0]$.
Let $\tilde{x} \in \baire$ be such that
 for every transitive model $N$ of $\zf$ such that
 $g_0 \in N$,
 we have $X \in N$ iff
 $\tilde{x} \in N$.

Force over $M[g_0]$ by
 $\mbb{H}^{M[g_0]}$ to get $g_1$ so that
 $\tilde{x} \le_T \tilde{a} \oplus (\bigcup \bigcap g_1)$.
Let $G := g_0 * g_1$,
 so $G \in \tilde{V}$ is $\mbb{P}$-generic over $M$.
$L(A,G)$ is a model of $\zf$ and it contains
 $g_0$ and $A$, so it contains $\tilde{a}$.
It also contains $g_1$, therefore
 it contains $\tilde{x}$.
Since it contains $g_0$ and $\tilde{x}$,
 it contains $X$.
\end{proof}

Note that if $\lambda = \omega$ in the theorem above,
 then we can simply take $\mathbb{P}$ to be $\mathbb{H}^M$.

\section{Proof of Generic Coding with Help Theorem}

Theorem~\ref{main_thm} follows
 from the Main Lemma of
 \cite{Hathaway1}.
For completeness we give a full proof here.

\subsection{Evasiveness and the Sticking Out Lemma}

We will now start to prove the theorem.
This subsection helps to clarify
 how we use
 the hypothesis
 $\bar{a} \not\in M$.


\begin{definition}
Let $M$ be a transitive model of $\zf$.
A set $A \subseteq \omega$ is
 \textit{evasive with respect to $M$}
 iff it is infinite and it has no infinite subsets
 in $M$.
\end{definition}


\begin{fact}
\label{comp_from_inf_subsets}
Given any $\bar{a} \in \baire$,
 there is a set $A \subseteq \omega$
 such that $\bar{a} =_T A$ and
 $A$ is computable from every infinite
 subset of itself.
\end{fact}

Thus if $M$ is a transitive model of $\zf$
 and $\bar{a} \in \baire - M$,
 then if $A$ comes from the fact above,
 then $A$ is evasive with respect to $M$.

\begin{lemma}[Sticking Out Lemma]
\label{sticking_out_lemma}
Let $M$ be a transitive model of $\zf$.
Let $A \subseteq \omega$ be evasive with respect to $M$.
Then if $B \subseteq \omega$ is
 infinite and in $M$,
 then $B - A$ is infinite.
\end{lemma}
\begin{proof}
Assume towards a contradiction that
 $B - A$ is finite.
Then $B - A \in M$.
Since both $B$ and $B - A$ are in $M$,
 we have $B \cap A \in M$ as well.
At the same time,
 since $B$ is infinite
 and $B - A$ is finite,
 $B \cap A$ must be infinite.
So now we have shown that
 $B \cap A$ is an infinite subset of $A$
 that is in $M$,
 which contradicts $A$
 being evasive with respect to $M$.
\end{proof}


\subsection{Decoding an $x \in \baire$
 from an $\mbb{H}$ generic and an $A \subseteq \omega$}

Suppose $G$ is generic for $\mbb{H}$.
Recall that $g := \bigcup \bigcap G$
 is a function from $\omega$ to $\omega$.
The idea is to look at
 each $n \in \omega$ such that $g(n) \in A$.
Which element of $A$ this $g(n)$
 actually is will give us a piece of encoded
 information.
For each $n$ such that $g(n) \not\in A$,
 no information is being encoded.
Here is what we mean precisely:

\begin{definition}
Fix a computable function
 $\theta : \omega \to \omega$ such that
 $$(\forall m \in \omega)\, \theta^{-1}(m)
 \mbox{ is infinite. }$$
Given an infinite $A \subseteq \omega$,
 let $e_A : \omega \to A$ be the strictly increasing
 enumeration of $A$.
Let $\eta_A : A \to \omega$ be the function
 $\theta \circ e_A^{-1}$.
\end{definition}
Note that for each $m \in \omega$,
 $\eta_A^{-1}(m) \subseteq A$ is infinite.

\begin{definition}
Let $M$ be a transitive model of $\zf$.
Let $G$
 be $\mbb{H}^M$-generic over $M$.
Let $A \subseteq \omega$.

Then the real that is
 $A$-\defemp{encoded} by $G$ is
 $$\langle \eta_A( g(n_i) ) : i < \omega \rangle,$$
 where $g := \bigcup \bigcap G$ and
 $$n_0 < n_1 < ...$$
 is the increasing enumeration of the
 set of $n \in \omega$
 such that $g(n) \in A$.
However,
 if there are only finitely many such $n$'s,
 then the real $A$-encoded by $G$
 is the zero sequence.
\end{definition}

\begin{observation}
\label{complexity_obs}
Let $x \in \baire$ be the real
 $A$-encoded by $G$.
Then $$x \le_T A \oplus (\bigcup \bigcap G).$$
\end{observation}

\subsection{The stronger $\le_A$ ordering
 and the Main Lemma}

Given $A \subseteq \omega$,
 there is an ordering $\le_A$
 defined on $\mbb{H}$
 which is stronger than $\le$.
Intuitively, $T' \le_A T$ iff
 $T' \le T$
 and the stem of $T'$ does not ``hit'' $A$
 any more than the stem of $T$ already does:
\begin{definition}
Let $A \subseteq \omega$.
Then given $t,t' \in {^{<\omega} \omega}$, we write
 $t' \sqsupseteq_A t$
 iff $t' \sqsupseteq t$ and
 $$(\forall n \in \dom(t') - \dom(t))\, t'(n) \not\in A$$
\end{definition}
\begin{definition}
Let $A \subseteq \omega$.
Given $T, T' \in \mbb{H}$,
 we write $T' \le_A T$ iff
  $T' \le T$ and
  $\mbox{Stem}(T') \sqsupseteq_A \mbox{Stem}(T)$.
\end{definition}

The content of the Main Lemma soon to come
 is that as long as
 $A$ is evasive with respect to $M$,
 we can hit dense subsets of $\mbb{H}$
 (that are in $M$) by making $\le_A$ extensions.
So, we can construct a generic without being
 forced to encode unwanted information.
Hence, we can alternate between
 1) making $\le_A$
 extensions in order to build an $\mbb{H}$
 generic but not encoding any information and
 2) making non-$\le_A$ extensions to encode information.
We use a \textit{rank analysis}
 to prove the Main Lemma:
\begin{definition}
\label{reachability}
Given $S \subseteq {^{<\omega}\omega}$ and
 $t \in {^{<\omega}\omega}$,
\begin{itemize}
\item $t$ is $0$-$S$-reachable iff $t \in S$;
\item $t$ is $\alpha$-$S$-reachable for some $\alpha > 0$ iff
 $$\{ z \in \omega : t ^\frown z
 \mbox{ is $\beta$-$S$-reachable for some }
 \beta < \alpha \}$$
 is infinite;
\item $t$ is $S$-reachable iff $t$ is
 $\alpha$-$S$-reachable for some $\alpha$.
\end{itemize}
\end{definition}

Notice that if $t$ is not $S$-reachable,
 then only a finite set of successors of
 $t$ can be $S$-reachable.

\begin{lemma}
\label{dense_implies_reachable}
Let $\mc{D} \subseteq \mbb{H}$ be dense.
Let $$S = \{ s \in {^{<\omega}\omega} :
 (\exists T \in \mc{D})\, \mbox{Stem}(T) = s \}.$$
Fix $t \in {^{<\omega}\omega}$.
Then $t$ is $S$-reachable.
\end{lemma}
\begin{proof}
Assume that some fixed $t$ is not $S$-reachable.
We will construct a tree $T \in \mbb{H}$
 with stem $t$ such that no $s \sqsupseteq t$ in $T$
 is in $S$.
Hence, no $T' \le T$ can be in $\mc{D}$.

There is only a finite set of $z \in \omega$
 such that $t ^\frown z$ is $S$-reachable.
Let the successors of $t$ in $T$ be those $t ^\frown z$
 that are not $S$-reachable.
Now for each $t ^\frown z_0$ in $T$,
 there is only a finite set of $z \in \omega$ such that
 $t ^\frown z_0 ^\frown z$ is $S$-reachable.
Let the successors of each $t ^\frown z_0$ in $T$ be those
 $t ^\frown z_0 ^\frown z$ that are not $S$-reachable.
Continuing this procedure $\omega$ times
 yields a tree $T$ such that all
 $s \sqsupseteq t$ in $T$
 are not $S$-reachable.
In particular no $s \sqsupseteq t$ in $T$
 is in $S$.
\end{proof}

\begin{lemma}[Main Lemma]
\label{main_lemma}
Let $M$ be a transitive model of $\zf$.
Let $A \subseteq \omega$ be evasive with respect to $M$.
Let $\mbb{P} = \mbb{H}^M$.
Let $\mc{D} \in \p^M(\mbb{P})$ be open dense
 (in $M$).
Let $T \in \mbb{P}$.
Then there exists some $T' \le_A T$ in $\mc{D}$.
\end{lemma}
\begin{proof}
Let $t = \mbox{Stem}(T)$.
Let $$S = \{ s \in {^{<\omega}\omega} :
 (\exists T' \in \mc{D})\, \mbox{Stem}(T') = s \}.$$
If we can find a $s \sqsupseteq_A t$ in $T \cap S$,
 then letting $T' \in \mc{D}$ be such that
 $\mbox{Stem}(T') = s$ and letting
 $T'' \le T$ be $T'' = T \restriction s$,
 then $T' \cap T''$ is in $\mc{D}$
  (because $\mc{D}$ is open),
 and $\mbox{Stem}(T' \cap T'') = s$ so
 $T' \cap T'' \le_A T$.
Hence, we will be done.

Now by the previous lemma,
 fix some ordinal $\alpha$ such that
 $t$ is $\alpha$-$S$-reachable.
If $\alpha = 0$ we are done, so assume $\alpha > 0$.
The set
 $$B = \{ z \in \omega : t ^\frown z
 \mbox{ is $\beta$-$S$-reachable for some }
 \beta < \alpha \}$$
 is infinite and in $M$.
Since $A$ is evasive with respect to $M$,
 $B - A$ must be infinite by the Sticking Out Lemma
 (Lemma~\ref{sticking_out_lemma}).
Thus, we may fix some
 $z_0 \in (B - A)$ such that $t ^\frown z_0 \in T$.

Now $t ^\frown z_0$ is $\beta$-$S$-reachable for some
 fixed $\beta < \alpha$.
If $\beta = 0$ we are done, and otherwise we may
 find some $z_1 \in (B-A)$,
 for an appropriately redefined $B$,
 such that $t ^\frown z_0 ^\frown z_1 \in T$ and
 $t ^\frown z_0 ^\frown z_1$ is $\gamma$-$S$-reachable
 for some $\gamma < \beta$.
We may continue like this but eventually we will have some
 $t ^\frown z_0 ^\frown ... ^\frown z_n$
 that is in $S$.
\end{proof}

\subsection{Proof of Generic Coding with Help Theorem}


\begin{proof}[Proof of Theorem~\ref{main_thm}]
Let $M$ be a transitive model of $\zf$.
Let $\mbb{P} = \mbb{H}^M$
 and assume $\p^M(\mbb{P})$ is countable.
Let $x \in \baire$.
Let $\bar{a} \in \baire - M$.
By Fact~\ref{comp_from_inf_subsets},
 fix $A \subseteq \omega$ such that
 $A =_T \bar{a}$ and $A$ is computable from every
 infinite subset of itself.
Then $A$ is evasive with respect to $M$.
It suffices to find a $\mbb{P}$-generic $G$ over $M$
 such that $x \le_T A \oplus (\bigcup \bigcap G)$.
By Observation~\ref{complexity_obs},
 it suffices to find a $\mbb{P}$-generic $G$ over $M$
 such that $x$ is the real $A$-encoded by $G$.

Since $\p^M(\mbb{P})$ is countable,
 let $\langle \mc{D}_i : i < \omega \rangle$
 be an enumeration
 of the open dense subsets of $\mbb{P}$ in $M$.
We will construct a decreasing $\omega$-sequence
 $$T_0 \ge T_1 \ge ...$$
 of $\mbb{P}$-conditions
 such that each $T_i \in \mc{D}_i$.
Hence $$G := \{ T \in \mbb{P} :
 (\exists i)\, T \ge T_i \}$$
 will be $\mbb{P}$-generic over $M$.
On the other hand, we will construct the sequence
 of conditions so that $x$ is the real $A$-encoded by $G$.

Since $A$ is evasive with respect to $M$,
 by Lemma~\ref{main_lemma}
 (the Main Lemma),
 let $T_0 \le_A 1_\mbb{P}$
 be such that $T_0 \in \mc{D}_0$.
Now we will encode $x(0)$:
 let $T_0' \le T_0$ be a non-$\le_A$ extension
 of $T_0$,
 extending the stem of $T_0$ by one, such that
 $\mbox{Stem}(T_0') = \mbox{Stem}(T_0) ^\frown z$
 for a $z \in A$ such that
 $\eta_A(z) = x(0)$.
This is possible because
 $$\{ z \in A : \eta_A(z) = x(0) \}$$
 is infinite, and so must intersect
 $$\{ z \in \omega :
 \mbox{Stem}(T_0)^\frown z \in T_0 \}.$$

Next, let $T_1 \le_A T_0'$ be such that
 $T_1 \in \mc{D}_1$.
Then, let $T_1' \le T_1$ be such that
 $\mbox{Stem}(T_1') = \mbox{Stem}(T_1) ^\frown z$
 for a $z \in A$ such that $\eta_A(z) = x(1)$.

Continuing this $\omega$ times,
 we see that $x$ is the real $A$-encoded
 by $G$.
That is, let $g := \bigcup \bigcap G$.
The only $n$'s such that
 $g(n) \in A$ come from when we made non-$\le_A$
 extensions.
And, if $n_0 < n_1 < ...$
 is the strictly increasing enumeration
 of these $n$'s,
 then we see that $\eta_A(g(n_i)) = x(i)$ for each $i$.
\end{proof}

\subsection{Another application of the Main Lemma}

As described in the introduction,
 here is the original kind of result for which the
 Main Lemma was created.
A proof can be found in \cite{Hathaway2}.
\begin{proposition}
\label{original_application}
Assume $\ad^+$.
Fix $a \in \baire$.
Then there is a Borel
 (in fact, Baire class one)
 function $f_a : \baire \to \baire$
 such that whenever
 $g : \baire \to \baire$ is a function
 whose graph is disjoint from $f_a$,
 then $$a \in L[C]$$
 where $C \subseteq \mbox{Ord}$
 is any $\infty$-Borel code for $g$.
\end{proposition}

The function $(a,x) \mapsto f_a(x)$
 is Borel as well.

\section{HOD}

By Vop\u{e}nka's Theorem,
 every real is generic over $\hod$.
But one can ask if there is a single
 $\mathbb{P} \in \hod$ such that
 $(|\mathbb{P}| \le 2^\omega)^\hod$
 and every real is $\mathbb{P}$-generic over $\hod$.
This is relevant to our paper because
 by Theorem~\ref{generically_generic},
 if $\tilde{V}$ is an outer model of $V$ in which
 $\p^{\hod^V}(\mathbb{H}^{\hod^V})$ is countable,
 and $\bar{a} \in (\baire)^{\tilde{V}} - \hod^V$ is arbitrary,
 then for any $x \in (\baire)^{\tilde{V}}$,
 there is a $G$ that is $\mathbb{H}^{\hod^V}$-generic over $\hod^V$
 such that $x \in L(\bar{a},G)$.
So the question is whether the $\bar{a}$ can be removed.
The answer is no:
\begin{proposition}
It is consistent with $\zfc$ that
 there is a real $R$ that is not $\mathbb{P}$-generic
 over $\hod$ for any $\mathbb{P} \in \hod$ such that
 $(|\mathbb{P}| \le 2^\omega)^\hod$.
Moreover, this persists to any outer model of $V$.
That is, if $\tilde{V}$ is an outer model of $V$,
 then $R$ is not $\mathbb{P}$-generic over $\hod^V$
 for any $\mathbb{P} \in \hod^V$ such that
 $(|\mathbb{P}| \le 2^\omega)^{\hod^V}$.
\end{proposition}
\begin{proof}
Start with $L$.
Let $\mathbb{C}_{\omega_2} \in L$ be the forcing to add a Cohen subset of $\omega_2$.
Let $A \subseteq \omega_2$ be $\mathbb{C}_{\omega_2}$-generic over $L$.
Let $X \subseteq \omega_1$ be generic over $L[A]$
 by almost disjoint coding such that
 $A \in L[X]$.
Let $R \subseteq \omega$ be generic over $L[X]$
 by almost disjoint coding such that
 $X \in L[R]$.
So now $$L \subseteq L[A] \subseteq L[X] \subseteq L[R]$$
 and $\p(\omega_1)^L = \p(\omega_1)^{L[A]}$
 (because $\mathbb{C}_{\omega_2}$ is
 ${<}\omega_2$-closed).
Let $H = \hod^{L[R]}$.
We will show that $L[R]$ satisfies that $R$ is not generic over
 $H$ by any forcing of size $(2^\omega)^H$.
Moreover, fix any outer model $N$ of $L[R]$.
We will show that $N$ satisfies that $R$ is not generic over
 $H$ by any forcing of size $(2^\omega)^H$.

The forcing $\mathbb{Q}$ to go from $L[A]$ to $L[R]$
 is weakly homogeneous \cite{ad_code}.
This is subtle, because a three step iteration of almost disjoint coding,
 to code a subset of $\omega_3$ into a subset of $\omega$,
 may $\textit{not}$ be weakly homogeneous \cite{ad_code}.
Now because $\mathbb{Q} \in L[A]$ is weakly homogeneous, $H \subseteq L[A]$.
 
Since $L \subseteq H \subseteq L[A]$
 and $\p(\omega_1)^L = \p(\omega_1)^{L[A]}$,
 we have $\omega_1^L = \omega_1^H = \omega_1^{L[A]}$
 and $H$ satisfies $\ch$.
Suppose towards a contradiction that there is some $\mathbb{P} \in H$
 such that $R$ is in a generic extension of $H$ by $\mathbb{P}$
 (meaning there is some $G \in N$ that is $\mathbb{P}$-generic over $H$ and $R \in H[G]$)
 and $\mathbb{P}$ has size $(2^\omega)^H = \omega_1^H$.
Then because $\p(\omega_1)^L = \p(\omega_1)^H$,
 a forcing $\tilde{\mathbb{P}}$ isomorphic to $\mathbb{P}$ is in $L$.
Also because $\p(\omega_1)^L = \p(\omega_1)^H$,
 all dense subsets of $\tilde{\mathbb{P}}$ in $H$ are already in $L$.
Let $G \subseteq \tilde{\mathbb{P}}$ in $N$ be
 $\tilde{\mathbb{P}}$-generic over $L$
 such that $R \in L[G]$.
Note that this implies $A \in L[G]$.

By a density argument for $\mathbb{C}_{\omega_2}$,
 the set $A \subseteq \omega_2^L$ has no
 subset of size $\omega_2^L$ in $L$.
On the other hand,
 let $\dot{A}$ be a $\tilde{\mathbb{P}}$ name for $A$.
For each $\alpha \in A$,
 let $p_\alpha \in G$ be a condition
 such that $p_\alpha \forces \check{\alpha} \in \dot{A}$.
Since $(|\tilde{\mathbb{P}}| < \omega_2)^L$,
 fix a $p \in G$ such that
 $p = p_\alpha$ for a size $\omega_2^L$ set of $\alpha \in A$.
Now the set
 $$\{ \alpha < \omega_2^L :
 p \forces \check{\alpha} \in \dot{A} \}$$
 is a size $\omega_2^L$ subset of $A$ in $L$,
 which is a contradiction.
\end{proof}

We mentioned in the proof above that
 the three step iteration of almost disjoint coding
 to code a subset of $\omega_3$ into a subset of $\omega$
 may not be weakly homogeneous.
The argument in the proof above
 also shows us why:
 start with $V = L$ and let $A \subseteq \omega_3$
 be a Cohen subset of $\omega_3$.
Let $R \subseteq \omega$ arise from the
 three step iteration $\mathbb{Q} \in L[A]$ of almost disjoint coding
 to code $A \subseteq \omega_3$ into a subset of $\omega$.
Suppose towards a contradiction that $\mathbb{Q}$ is weakly homogeneous.
Then $\hod^{L[R]} \subseteq L[A]$.
By Vop\u{e}nka's Theorem,
 $R$ is generic over $\hod^{L[R]}$
 by a forcing of size $\omega_2$.
Since $\p(\omega_2)^{L} = \p(\omega_2)^{L[A]}$
 and $\hod^{L[R]}$ is intermediate between $L$ and $L[A]$,
 there must be some $\tilde{\mbb{P}} \in L$ of size $\omega_2$
 such that $R$ is in a $\tilde{\mbb{P}}$-generic extension $L[G]$ of $L$.
But now since $A \in L[G]$ and $|\tilde{\mbb{P}}| \le \omega_2$,
 $A$ has a size $\omega_3$ subset in $L$.
This contradicts $A$ being Cohen generic over $L$.

\section{Questions}

\subsection{What can replace $\mbb{H}$?}

\begin{question}
Let $M$ be a c.t.m.\ of $\zfc$.
What are the forcings $\mathbb{P} \in M$ such that
 every real $a \in \baire - M$
 is $(\mathbb{P},M)$-helpful?
Does Cohen forcing work?
What about a forcing which is $^\omega \omega$-bounding?
\end{question}

\subsection{Generically coding subsets of $\omega_1$
 with help}

Given a transitive model $M$,
 it is natural to ask whether subsets of $\omega_1$
 can be coded by generics over $M$ with help.
By Theorem~\ref{generically_generic},
 this is possible as long as we pass
 to a sufficiently larger outer model $\tilde{V}$.
We suspect that passing to $\tilde{V}$ is not
 necessary provided that $M$ is large enough.
In terms of being large enough,
 note that given a forcing $\mbb{P} \in L(\mbb{R})$, if
\begin{itemize}
\item there is a surjection of $\baire$ onto
 $\mbb{P}$ in $L(\mbb{R})$,
\item $\mbb{P}$ is countably closed,
\item there is a proper class of Woodin cardinals, and
\item $\ch$ holds,
\end{itemize}
 then there is a $\mbb{P}$-generic over $L(\mbb{R})$
 in $V$.
Here is a proof of this fact
 (pointed out by Paul Larson):
 every set of reals in $L(\mbb{R})$
 is the continuous preimage of $\mathbb{R}^\#$,
 so there are at most $2^\omega$ sets of reals in $L(\mathbb{R})$.
But, because $\ch$ holds,
 there are $\omega_1$ sets of reals in $L(\mathbb{R})$.
So there are $\omega_1$ dense subsets of $\mathbb{P}$ in $L(\mathbb{R})$.
Let $\langle D_\alpha : \alpha < \omega_1 \rangle$ be an
 enumeration of all these dense sets.
By the fact that $\mbb{P}$ is countably closed,
 we can hit all $\omega_1$ dense sets by
 forming a length $\omega_1$ decreasing sequence through $\mathbb{P}$
 (here we also use that ${^{<{\omega_1}} \mbb{P}} \subseteq L(\mathbb{R})$).

So, we ask the following
 (where we have weakened arbitrary help
 $\bar{a} \in \p(\omega_1) - L(\mbb{R})$ to some fixed
 help $\bar{a} \in \p(\mbb{R})$):
\begin{question}
\label{omega_one}
Assume $\ch$ and a proper class of Woodin cardinals.
Is there some $\bar{a} \subseteq \mbb{R}$
 and some forcing $\mbb{P} \in L(\mbb{R})$ that is
 countably closed such that given any
 $X \subseteq \omega_1$,
 there is a $G$ that is $\mbb{P}$-generic over $L(\mbb{R})$
 such that $X \in L(\bar{a}, G, \mbb{R})$?
\end{question}

Along similar lines, Woodin has conjectured
 (Section 10.6 of \cite{Woodin})
 that assuming $\ch$ and a
 $\textit{measurable}$ Woodin cardinal,
 then for any $X \subseteq \omega_1$,
 there is some $B \subseteq \mbb{R}$
 such that $L(B,\mbb{R}) \models \ad^+$
 and $X \in L(B,\mbb{R})[G]$
 for some $G$ that is
 $\mbox{Col}(\omega_1, \mbb{R})$-generic over
 $L(\mbb{R},B)$.

Assume Woodin's conjecture is true
 and assume $V$ satisfies $\ch$ and has a measurable Woodin cardinal.
Let $\mc{C}$ be the collection of all
 inner models of $\ad^+$ containing all the reals.
Then every subset of $\omega_1$
 (and therefore every subset of $\mathbb{R}$ because we are assuming $\ch$)
 is generic over some model in $\mc{C}$.
Our question above asks whether the smallest model in $\mc{C}$,
 namely $L(\mathbb{R})$,
 is still large enough so that
 $(\exists \bar{a} \subseteq \mathbb{R})
  (\exists \mathbb{P} \in L(\mathbb{R}))
  (\forall X \subseteq \omega_1)
  (\exists G$ that is $\mbb{P}$-generic over $L(\mbb{R}))\, X \in L(\bar{a},G,\mathbb{R}).$

\section{Addendum}

In \cite{coding_the_universe},
 it is shown (in $\zfc$) that
 there exists a real $c$ that is class generic over $V$
 such that $V \subseteq L[c]$
 and $\omega_1^L < \omega_1$ in the extension.
Consider the model $V[c]$.
In this model, there are two reals $g_1, g_2$
 both Cohen generic over $L$
 such that $c \le_T g_1 \oplus g_2$.
Hence, there are two set generic extensions of $L$ (in $V[c]$)
 that together code all of $V$.
We observe that the Generic Coding with Help Theorem
 gives us a related result:
\begin{theorem}
Let $M$ be any transitive model of $\zf$.
Let $\bar{a}$ be a set of ordinals not in $M$
 but such that $\mbox{sup}( \bar{a} ) \in M$.
Then there is a $G$ that is set generic over $M$
 (which exists in a class forcing extension of $V$)
 such that $V \subseteq L[G][\bar{a}]$.
\end{theorem}
\begin{proof}
Let $\lambda$ be the ordinal $\mbox{sup}( \bar{a} )$.
Let $\mathbb{P} = (\mbox{Col}(\omega,\lambda)*
 \mathbb{H})^M$.
Let $\tilde{V}$ be a set forcing extension
 of $V$ in which $\mc{P}^M(\mathbb{P})$
 is countable.
Let $c$ be a real class generic over $\tilde{V}$
 such that $\tilde{V} \subseteq L[c]$.
Apply Theorem~\ref{generically_generic}
 in $\tilde{V}[c]$
 to get a $G$ that is $\mathbb{P}$-generic over $M$
 such that $c \in L[G][\bar{a}]$.
So we have
 $$V \subseteq \tilde{V} \subseteq L[c]
 \subseteq L[G][\bar{a}].$$
\end{proof}

Trying to find a $G$ such that $V = L[G][\bar{a}]$
 is impossible in general, because $V$
 may not be of the form $L[X]$ for any set $X$.
However, note that if $V = L[r]$ for
 a real $r$
 and if $M$ is any transitive model of $\zf$
 such that $\mathbb{R}^M$ is countable
 and if $\bar{a}$ is any \textit{real} not in $M$,
 then we \textit{can} find a $G$ that is set generic
 over $M$ such that $V = L[G][\bar{a}]$.
To do this,
 use the Generic Coding with Help Theorem
 to get a real $G$ (in $V$)
 that is $\mathbb{H}^M$-generic over $M$
 such that $r \in L[G][\bar{a}]$.
Now $V = L[r] \subseteq L[G][r] \subseteq L[G][\bar{a}]
 \subseteq V$.

\section{Acknowledgements}

We would like to thank
 Andreas Blass for discussing the
 Generic Coding with Help method.
We would also like to thank Paul Larson
 for pointing out situations where there
 are generics over $L(\mbb{R})$.

\end{document}